\documentclass[11pt]{amsart}
\usepackage{lmodern}
\usepackage{amsmath, amsthm, amssymb, amsfonts}
\usepackage[normalem]{ulem}
\usepackage{hyperref}
\usepackage{enumitem} 

\usepackage{verbatim} 
\usepackage{longtable}

\usepackage{mathtools}
\DeclarePairedDelimiter\ceil{\lceil}{\rceil}

\usepackage{tikz}
\usetikzlibrary{decorations.pathmorphing}
\tikzset{snake it/.style={decorate, decoration=snake}}

\usepackage{caption}

\usepackage{tikz-cd}
\usetikzlibrary{arrows}

\theoremstyle{plain}
\newtheorem{thm}{Theorem}[section]
\newtheorem{cor}[thm]{Corollary}
\newtheorem{lem}[thm]{Lemma}
\newtheorem{prop}[thm]{Proposition}

\theoremstyle{definition}
\newtheorem{defn}[thm]{Definition}

\theoremstyle{remark}
\newtheorem{rmk}[thm]{Remark}

\newcommand{\BA}{{\mathbb{A}}}

\newcommand{\BC}{{\mathbb{C}}}

\newcommand{\BH}{{\mathbb{H}}}

\newcommand{\BP}{{\mathbb{P}}}
\newcommand{\BQ}{{\mathbb{Q}}}
\newcommand{\BR}{{\mathbb{R}}}

\newcommand{\BZ}{{\mathbb{Z}}}

\newcommand{\CA}{{\mathcal A}}

\newcommand{\CD}{{\mathcal D}}

\newcommand{\CM}{{\mathcal M}}
\newcommand{\CN}{{\mathcal N}}

\newcommand{\CP}{{\mathcal P}}

\newcommand{\Fp}{{\mathfrak{p}}}

\DeclareFontFamily{OT1}{rsfs}{}
\DeclareFontShape{OT1}{rsfs}{n}{it}{<-> rsfs10}{}
\DeclareMathAlphabet{\curly}{OT1}{rsfs}{n}{it}

\newcommand\Spec{\operatorname{Spec}}

\usepackage{tikz}
\usepackage{lmodern}
\usetikzlibrary{decorations.pathmorphing}

\title{The $P=W$ identity for isolated cluster varieties: full rank case}
\author{Zili Zhang}
\address{School of Mathematical Sciences, Tongji University, Shanghai, China}
\email{zhangzili@tongji.edu.cn}
\date{\today}

\begin{document}

\maketitle
\begin{abstract}
We initiate a systematic construction of real analytic Lagrangian fibrations from integer matrices. We prove that when the matrix is of full column rank, the perverse filtration associated with the Lagrangian fibration matches the mixed Hodge-theoretic weight filtration of the isolated cluster variety associated with the matrix. 
\end{abstract}
\tableofcontents

\section{Introduction}

\subsection{Perverse filtrations}
Let $Y$ be a complex algebraic variety. The bounded derived category $D_c^b(Y)$ of constructible sheaves of $\BQ$-vector spaces on $Y$ is naturally equipped with a perverse $t$-structure $({^p\CD}^{\le 0},{^p\CD}^{\ge 0})$. For a smooth morphism $f:X\to Y$ with relative dimension $d$, applying the perverse truncations 
\[
^p\tau_{\le k}: D_c^b(Y)\to {^p\CD}^{\le k} ,~~k\in\BZ
\]
on $Rf_*\BQ_X$ defines an increasing filtration
\begin{equation}\label{00}
P_0H^*(X,\BQ)\subset P_1H^*(X,\BQ)\subset\cdots H^*(X,\BQ).
\end{equation}
The filtration (\ref{00}) is called the perverse filtration associated with the morphism $f$. De Cataldo and Migliorini proved in \cite{dCM} that perverse filtrations satisfy a Lefschetz type symmetry: any relatively ample class $\alpha\in H^2(X,\BQ)$ induces an isomorphism
\[
\textup{Gr}^P_{d-k} H^*(X,\BQ)\xrightarrow[\cong]{\cup\alpha^k}\textup{Gr}^P_{d+k} H^*(X,\BQ),~~k\ge0,
\]
called the relative hard Lefschetz property.

\subsection{The P=W phenomenon}
Let $C$ be a smooth projective curve. The moduli spaces of geometric objects on $C$ are naturally endowed with rich algebro-geometrical and topological structures. The character variety $\CM_B$, the moduli space of all $\textup{GL}(n,\BC)$-representations of the fundamental group $\pi_1(C)$, is a affine scheme. The Dolbeault moduli space $\CM_D$, the moduli space of all degree 0 rank $n$ semi-stable Higgs bundles on $C$, is naturally equipped with a Hitchin map $h:\CM_D\to\BA$ onto an affine space. Simpson proves in \cite{Sim} that there exists a canonical homeomorphism $\CM_B\cong\CM_D$ between the quasi-projective varieties, called the non-abelian Hodge correspondence. Therefore, there is a canonical identification of the cohomology groups $H^*(\CM_D,\BQ)=H^*(\CM_B,\BQ)$. Under this identification, one may compare the mixed Hodge theoretic weight filtration on $H^*(\CM_B,\BQ)$ and the perverse filtration on $H^*(\CM_D,\BQ)$ associated with the Hitchin map. The following mysterious $P=W$ identity is conjectured by de Cataldo, Hausel and Migliorini in \cite{dCHM},  proved now by Maulik-Shen in \cite{MS} and Hausel-Mellit-Minets-Schiffmann in \cite{HMMS} independently via different methods.
\begin{thm}[$P=W$]
Under the canonical identification $H^*(\CM_B,\BQ)=H^*(\CM_D,\BQ)$ via the non-abelian Hodge theory, the $P=W$ identity holds:
\[
P_kH^*(\CM_D,\BQ)=W_{2k}H^*(\CM_B,\BQ)=W_{2k+1}H^*(\CM_B,\BQ),~~k\ge0.
\]
\end{thm}

Where to expect a $P=W$ identity is still a widely open problem. However, when the $P=W$ identity holds, the relative hard Lefschetz symmetry (\ref{00}) induces via the non-abelian Hodge correspondence the so called the curious hard Lefschetz property on the character varieties. A variety $X$ satisfies the curious hard Lefschetz property if there exists a 2-form of mixed Hodge type (2,2) such that
\begin{equation} \label{chl}
\textup{Gr}^W_{2\dim X-2k} H^*(X,\BQ)\xrightarrow[\cong]{\cup\beta^k}\textup{Gr}^W_{2\dim X+2k} H^*(X,\BQ),~~k\ge0.
\end{equation} 

Since the curious hard Lefschetz symmetry is rarely satisfied, it is a good indicator to suggest potential $P=W$ phenomena. Speyer and Lam proved that the even-dimensional Louise type cluster varieties of full rank satisfy the curious hard Lefschetz property (\ref{chl}), which indicates that the $P=W$ identity may hold for such cluster varieties and certain Hitchin-type fibrations. In \cite{Z}, the author proved that for each 2-dimensional cluster variety, the desired fibration does exist, and the $P=W$ identity holds. In this paper, we focus on a special type of higher dimensional cluster varieties, called the isolated cluster variety (see Section 2.1 for definition). Isolated cluster varieties play crucial role in the definition of cluster varieties of Louise type; an effective algorithm of iterated Mayer-Vietoris arguments starting from isolated cluster varieties will calculate cohomology groups of all cluster varieties of Louise type. We give a systematic construction of real Lagrangian fibrations associated with the isolated cluster varieties, and prove that: 

\begin{thm}[Theroem \ref{thm}]
Let $M$ be a full rank $m\times n$ integer matrix, where $m\ge n$. Let $X(M)$ be the isolated cluster variety and $Y(M)\to \BR^{n+m}$ be the associated Lagrangian fibration. Then the $P=W$ identity holds, i.e.
    \[
    P_kH^*(Y(M),\BQ)=W_{2k} H^*(X(M),\BQ)=W_{2k+1}H^*(X(M),\BQ),~~ k\ge 0. 
    \]
\end{thm}

\subsection{Outline}
In Section 2, we review basic facts about cluster varieties and perverse filtrations, and prove some linear algebra facts which are used later to study the structures of isolated cluster varieties. In Section 3, we construct and study the geometry of the isolated cluster varieties and Lagrangian fibrations, and finally prove the $P=W$ identity.

\subsection{Acknowledgements} 
I thank Thomas Lam, David Speyer, and Xiping Zhang for helpful discussions. I am partially supported by the Fundamental Research Funds for the Central Universities.

\section{Preparatory Results}
\subsection{Cluster varieties}
An extended exchange matrix $B=(b_{ij})$ is a $(n+m)\times n$ matrix with integer coefficients such that the top $n\times n$ minor is skew-symmetric, \emph{i.e.} $b_{ij}=-b_{ij}$ for $1\le i,j\le n$. For $1\le k\le n$, we define $\mu_k(B)=(b'_{ij})$, the mutation of $B$ in the direction $k$, as the following:
\[
b'_{ij}=\begin{cases}
-b_{ij} & \textup{if }i=k\textup{ or }j=k,\\
b_{ij}+b^+_{ik}b^+_{kj}-b^-_{ik}b^-_{ij}& otherwise,
\end{cases}
\]
where $x^+=\max\{x,0\}$ and $x^-=-\min\{x,0\}$.

A seed $t=(\textbf{x},B)$ consists of $n+m$ variables $x_1,\cdots,x_{n+m}$ and an extended exchange matrix matrix $B$. The mutation of seeds $(\textbf{x},B)$ in the direction $1\le k\le n$ is a new seed $\mu_k(t)=(\textbf{x}',\mu_k(B))$, where $x'_i=x_i$ for $i\neq k$ and 
\begin{equation} \label{c1}
x'_k=\frac{\prod_i x_i^{b_{ik}^+}+ \prod_i x_i^{b_{ik}^-}}{x_k}.
\end{equation}
We continue mutating the seeds in all possible directions, producing new seeds. The $(n+m)$-tuples $(x_1,\cdots,x_{n+m})$ obtained in this way are called clusters, and $x_i$ are called cluster variables. The cluster variables $x_1,\cdots,x_n$ change in different clusters, so they are called mutable variables. The cluster variables $x_{n+1},\cdots,x_{n+m}$ are the same in all clusters and are called frozen.

It follows from (\ref{c1}) that any cluster variable obtained from successive mutating from $(\textbf{x},B)$ are in the field $\BC(x_1,\cdots,x_{n+m})$. The $\BC$-subalgebra of $\BC(x_1,\cdots,x_{n+m})$ generated by all cluster variables and inverse of the frozen variables is the cluster algebra $A(\textbf{x},B)$ or simply $A(B)$. The cluster variety is $\CA(\textbf{x},B)=\Spec\,A(\textbf{x},B)$. In general, cluster algebra is not necessarily finitely generated. Various finiteness conditions are introduced in \cite{BFZ,M,LS}. The finiteness conditions are stated more naturally in the language of the graph construction of the cluster varieties, which is equivalent to the construction by matrices. 

Let $\vec{\Gamma}(t)=\vec{\Gamma}(B)$ be a directed graph with $n+m$ vertices $1,\cdots,n+m$, and $b_{ij}$ parallel directed edge $i\to j$ whenever $b_{ij}>0$. We put the cluster variable $x_i$ on vertex $i$. For a directed edge $e:i\to j$, we denote $h(e)=i$ and $t(e)=j$ the head and the tail of the edge. When we mutate the cluster algebra in the direction $k$, we proceed by the following steps:
\begin{enumerate}
    \item Replace $x_k$ by $x'_k=x_k^{-1}(\prod_{h(e)=i}x_{t(e)}+\prod_{t(e)=i}x_{h(e)})$.
    \item For each pair of edges $e,f$ such that $t(e)=h(f)=k$, we add an edge $i\to j$.
    \item Reverse all edges connected to $k$.
    \item Whenever there is a pair of opposite edges, \emph{i.e.} $h(e)=t(f)$ and $h(f)=t(e)$, cancel them. 
\end{enumerate}
It is easy to see that after the mutation in direction $k$ we obtain the graph $\vec{\Gamma}(\mu_k(t))$, as desired.  Let $\vec{\Gamma}_{red}$ be the graph obtained by replacing parallel edges by a single edge. When $\vec{\Gamma}(t)_{red}$ does not contain any directed cycle for some seed $t$, the cluster algebra $A(t)$ is called an acyclic cluster algebra. It is proved in \cite{BFZ} that acyclic cluster algebra is finitely generated.  The following proposition gives a concrete description of acyclic cluster variety, and  shows that for an acyclic variety, mutating each mutable variable once is sufficient to obtain all cluster variables.

\begin{prop}[\cite{BFZ}]
Let $\CA(\textbf{x},B)$ be an acyclic cluster variety, with extended exchange matrix $B=(b_{ij})_{(n+m)\times n}$. Then $\CA(\textbf{x},B)$ is the $(n+m)$-dimensional subvariety of $\BC^{2n}\times (\BC^*)^m$ cut out by the equations
\[
x_jx_j'=\prod_{i=1}^{n+m} x_i^{b_{ij}^+}+\prod_{i=1}^{m+n} x_i^{b_{ij}^-},~~ 1\le j\le n.
\]
\end{prop}

When we freeze a mutable variable $x_i$ of a given cluster variety, then by definition we obtain the cluster algebra $A_{\{i\}}:=A[x_i^{-1}]$. Similarly, we can freeze any subset $S\subset \{1,\cdots,n\}$ to get the cluster algebra $A_S$ and the cluster variety $\CA_S$. Let $e:i\to j$ be an edge connecting two mutable vertices. Then $e$ is called a separating edge if $e$ does not fit into a bi-infinite directed path of edges. It is easy to see that every edge in an acyclic graph is a separating edge. 

\begin{prop}[\cite{M}]
Suppose $e:i\to j$ is a separating edge. Then $x_i$ and $x_j$ are not simultaneously zero, i.e. the cluster variety $\CA$ is covered by $\CA_{\{i\}}$ and $\CA_{\{j\}}$.
\end{prop}

In \cite{LS}, Lam and Speyer introduce a weaker condition for cluster variety to be of finite type, called the Louise property. A cluster variety satisfies the Louise property if either 
\begin{enumerate}
   \item There are no edges among mutable vertices in $\vec{\Gamma}(t)_{red}$, or
   \item There exists a separating edge $e:i\to j$ such that the cluster varieties $\CA_{\{i\}}$, $\CA_{\{j\}}$, and $\CA_{\{i,j\}}$ all satisfy the Louise property. 
\end{enumerate}

When condition (1) of the Louise property  holds, the cluster variety is called an isolated cluster variety. It is obvious that isolated cluster varieties are acyclic, and acyclic cluster varieties always satisfy the Louise property. We remark that the recursive nature of the definition implies that cluster varieties satisfying Louise property are obtained as inductively glueing various isolated cluster varieties.

\subsection{Perverse filtrations}
Let $Y$ be a real analytic manifold. Let $D^b_c(Y)$ be the bounded derived category of constructible sheaves of $\BQ$-vector spaces on $Y$. The perverse $t$-structure on $D^b_c(Y)$ \footnote{We choose the perversity function as $p(n)=\ceil{-n/2}$, which matches the middle perversity when $Y$ is an complex algebraic variety.} has a perverse truncation functor 
\[
^p\tau_{\le k}:D^b_c(Y)\to D^b_c(Y)
\]
and a natural morphism
\[
^p\tau_{\le k}K\to K
\]
in $D^b_c(Y)$. Let $f:X\to Y$ be a proper morphism between real analytic manifolds. Then there is a natural morphism
\[
^p\tau_{\le k}Rf_*\BQ_X\to Rf_*\BQ_X, ~~k\in\BZ.
\]
Taking the hypercohomology yields a map
\begin{equation} \label{p}
\BH^d(Y,{^p\tau}_{\le k}Rf_*\BQ_X)\to \BH^d(Y,Rf_*\BQ_X)=H^d(X,\BQ)
\end{equation}
we define the perverse filtration $P_kH^d(X,\BQ)$ the image of (\ref{p}). The perverse filtration is an increasing filtration. Similar to \cite[Section 1.4.1]{dCHM}, we shift the indices of the perverse filtration such that for any proper morphism $f:X\to Y$ with equidimensional fiber, the fundamental class $1\in P_0H^0(X)$.

If a nonzero class $\alpha\in H^*(X,\BQ)$ satisfies $\alpha\in P_kH^*(X,\BQ)$ and $\alpha\not\in P_{k-1}H^*(X,\BQ)$, we denote $\Fp(\alpha)=k$.

We say a proper morphism $f:X\to Y$ admits a perverse decomposition if there exists suitable perverse sheaves $\CP_j$ on $Y$ such that 
\begin{equation}\label{0101}
Rf_*\BQ_X=\bigoplus_j \CP_j[-j].
\end{equation}
Such a decomposition always exists for proper morphisms of K\"ahler manifolds \cite[Theorem 0.6]{S}, but not for real analytic manifolds. When a perverse decomposition exists, then the perverse filtration can be computed as
\[
P_kH^*(X,\BQ)\cong\BH^*\left(\bigoplus_{j\le k}\CP_j[-j]\right).
\]
We remark that the perverse filtration is canonically defined, but perverse decomposition is not canonical. For proper morphisms which admit perverse decompositions, the K\"unneth property holds for the perverse filtration associated with the Cartesian product:

\begin{prop} \label{kunneth}
Let $f_i:X_i\to Y_i, 1\le i\le n$ be proper morphisms between real analytic manifolds. Suppose for each $i$, $f_i$ admits a perverse decomposition, then the perverse filtration of the Cartesian product $f:X_1\times\cdots\times X_n\to Y_1\times\cdots\times Y_n$ is
\[
P_kH^*(X_1\times\cdots\times X_n,\BQ)=\langle\alpha_1\boxtimes\cdots\boxtimes\alpha_n\mid \Fp(\alpha_1)+\cdots+\Fp(\alpha_n)\le k\rangle.
\]
\end{prop}

\begin{proof}
    The proof of \cite[Proposition 2.1]{Z1}, which is originally for proper morphisms of quasi-projective varieties, depends only on the existence of a perverse decomposition for each factor, and hence applies in our situation. 
\end{proof}

\begin{rmk}
   It's the author's ignorance whether the existence of perverse decomposition is necessary. Nevertheless, K\"unneth property under this restriction is enough for the purpose of this paper.
\end{rmk}

\begin{prop} \label{quotient}
Let $f:X\to Y$ be a proper morphism between real analytic manifolds. Let $G$ be a finite abelian group acting on $X$ properly discontinuously by analytic diffeomorphisms. Suppose $f:X\to Y$ descends to the quotient $X/G\to Y$. Then under the natural isomorphism
\[
H^*(X/G,\BQ)\cong H^*(X,\BQ)^G,
\]
the perverse filtration is identified as
\[
P_kH^*(X/G,\BQ)\cong (P_kH^*(X,\BQ))^G,
\]
\end{prop}

\begin{proof}
Consider the covering map $h:X\to X/G$. Since $h$ is finite, $Rh_*$ is just the ordinary pushforward of sheaves. Then the $G$ action on $h_*\BQ_X$ induces a decomposition indexed by characters of $G$
\begin{equation}\label{a1}
h_*\BQ_X=\bigoplus_{\chi\in\hat{G}}L_\chi,
\end{equation}
and the trivial character corresponds to the quotient $\BQ_{X/G}=(h_*\BQ_X)^G$. We write
\[
h_*\BQ_X=\BQ_{X/G}\oplus F.
\]
Then apply pushforward along $f'$ and perverse truncation functors
\[
\begin{tikzcd}
^p\tau_{\le k}Rf_*\BQ_X\arrow[r,equal]\arrow[d] &^p\tau_{\le k}Rf'_*\BQ_{X/G}\oplus {^p\tau_{\le k}}Rf'_*F \arrow[d]\\
Rf_*\BQ_X\arrow[r,equal]& f'_*\BQ_{X/G}\oplus f'_*F.
\end{tikzcd}
\]
Therefore 
\[
\begin{tikzcd}
  P_kH^*(X,\BQ)\arrow[r,equal]\arrow[d]&P_kH^*(X/G,\BQ)\oplus P_kH^*(Y,Rf'_*F)\arrow[d]\\
  H^*(X,\BQ)\arrow[r,equal]&H^*(X/G,\BQ)\oplus H^*(X/G,F)\\
\end{tikzcd}
\]
Since $F$ is the direct sum of nontrivial characters in the decomposition (\ref{a1}), $H^*(X/G,F)^G=0$. So we have $P_kH^*(X/G,\BQ)=(P_kH^*(X,\BQ))^G$.
\end{proof}

\subsection{Some linear algebra}
In this section, we include some facts in linear algebra which will be used later.

Let $T=(t_{ij})_{1\le i,j\le m}$ be a $m\times m$ integer matrix of full rank.  Define
\[
\begin{array}{cccc}
\CP(T):&(\BC^*)^m &\to& (\BC^*)^m\\ 
 &(\lambda_1,\cdots,\lambda_m) &\mapsto&\displaystyle\left(\prod_{j=1}^m\lambda_j^{t_{1j}},\cdots,\prod_{j=1}^m\lambda_j^{t_{mj}}\right).
\end{array}
\]
Let $\CN(T)$ be the kernel of $\CP(T)$. We have 

\begin{prop}
Let $T=(t_{ij})_{1\le i,j\le m}$ be an $m\times m$ integer matrix. Then $\CN(T)$ is isomorphic to $\BZ^m/T\BZ^m$. In particular, the abelian group $\CN(T)$ is finite if and only if $T$ is of full rank.
\end{prop}

\begin{proof}
By the definition of $\CP(T)$, the following diagram commutes:
\[
\begin{tikzcd}
0\arrow[r]&\BZ^m\arrow[r]\arrow[d,"T_\BZ"]&\BC^m\arrow[r,"\exp(2\pi i\cdot)"]\arrow[d,"T_\BC"]&(\BC^*)^m\arrow[r]\arrow[d,"\CP(T)"]&1\\
0\arrow[r]&\BZ^m\arrow[r]&\BC^m\arrow[r,"\exp(2\pi i\cdot)"]&(\BC^*)^m\arrow[r]&1
\end{tikzcd}
\]
Since $T_\BC$ is an isomorphism, the snake lemma yields an isomorphism $\CN(T)=\textup{coker}\, T_\BZ=\BZ^m/T\BZ^m$, which is finite if and only if $T$ is of full rank.
\end{proof}

\begin{prop} \label{lemma}
Let $M$ be a full rank $m\times n$ matrix with integer entries, where $m\ge n$. Then there exists an integer $d>0$, and two full rank $m\times m$ integer matrices $T$ and $\bar{M}$ such that $T\bar{M}=\left(\begin{array}{cc}dI_n&0\\0&I_{m-n}\end{array}\right)$, where $I_k$ is the identity matrix of $k\times k$, and the first $n$ columns of $\bar{M}$ is $M$.
\end{prop}

\begin{proof}
Since $M$ is of full rank and $m\ge n$, the column vectors of $M$ are linearly independent, and hence there exists a sequence of invertible row operations with integer coefficients to convert $M$ into $\binom{U}{0}$, where $U$ is an upper-triangular integer matrix. Equivalently, there exists a matrix $R\in\textup{SL}(m,\BZ)$ such that $RM=\binom{U}{0}$. Noting that $U$ is a square matrix of full rank, let $d$ be the greatest common divisor of all entries of $U^{-1}$. Now let $T=\left(\begin{array}{cc}dU^{-1}&0\\0&I_{m-n}\end{array}\right)R$ and $\bar{M}=R^{-1}\left(\begin{array}{cc}U&0\\0&I_{m-n}\end{array}\right)$.
\end{proof}

\begin{prop} \label{2.7}
Let $M$, $\bar{M}$, $T$ and $d$ as in Proposition \ref{lemma}. Then $\CN(T^T)$ is naturally a subgroup of $\CN\left(\begin{array}{cc}dI_n&0\\0&I_{m-n}\end{array}\right)=(\BZ/d\BZ)^n$. Furthermore, $\CN(T^T)$ depends only on $M$ and $d$, not on the choice of $T$. 
\end{prop}

\begin{proof}
   By Proposition \ref{lemma}, we have $\bar{M}^TT^T=\left(\begin{array}{cc}dI_n&0\\0&I_{m-n}\end{array}\right)$.   So the left multiplications induce natural morphisms 
   \[
   \BZ^m\xrightarrow{T^T}\BZ^m\xrightarrow{\bar{M}^T}\BZ^m
   \]
   which are injections since $T$ and $\bar{M}$ are of full rank. Therefore there is a natural inclusion 
   \[
   \CN(T^T)=\BZ^m/T^T\BZ^m\hookrightarrow \BZ^m/\left(\begin{array}{cc}dI_n&0\\0&I_{m-n}\end{array}\right)\BZ^m=\CN\left(\begin{array}{cc}dI_n&0\\0&I_{m-n}\end{array}\right).
   \]
   To show $\CN(T^T)=\BZ^m/T^T\BZ^m$ depends only on $d$ and $M$, it suffices to show that the $\BZ$-span of the columns of $T^T$ does not depend on the choice of $T$. As in Proposition \ref{lemma}, we choose and fix $R\in \textup{SL}(m,\BZ)$ such that $RM=\binom{U}{0}$, where $U$ is an full rank $n\times n$ integer matrix. Then $T\bar{M}=\left(\begin{array}{cc}dI_n&0\\0&I_{m-n}\end{array}\right)$ implies
   \[
   TR^{-1}\cdot R\bar{M}=\left(\begin{array}{cc}dI_n&0\\0&I_{m-n}\end{array}\right)
   \]
   Let $TR^{-1}=\left(\begin{array}{cc} A_{11}& A_{12}\\A_{21}&A_{22}\end{array}\right)$ and  $R\bar{M}=\left(\begin{array}{cc} U & B_{12}\\0&B_{22}\end{array}\right)$.
   An easy block matrix calculation implies  $A_{11}=dU^{-1}$, $A_{21}=0$, and $B_{22}=A_{22}^{-1}$. In particular, $A_{22}$ is an invertible matrix, and hence the row $\BZ$-span of $A_{22}$ is $\BZ^{m-n}$. Therefore, the row $\BZ$-span of $TR^{-1}$ is the direct sum of row $\BZ$-span of $dU^{-1}$ and $\BZ^{m-n}$. For different choice of $T$, $A_{12}$ and $A_{22}$ may vary but the the row span of $TR^{-1}$ are canonical. Now since $R\in \textup{SL}(n,\BZ)$, the injections 
   \[
   \BZ^m\xrightarrow{T^T}\BZ^m\xrightarrow{(R^T)^{-1}} \BZ^m,
   \]
   induces 
   \[
   \BZ^m/T^T\BZ^m\xrightarrow{\cong}\BZ^m/(R^T)^{-1}T^T\BZ^m\cong \BZ^n/dU^{-1}\BZ^n\oplus\BZ^{m-n}
   \]
   Note that the first isomorphism depends only on the choice of $R$, which is fixed and hence does not depends on the choice of $T$. The second isomorphism follows from the description of column $\BZ$-span of $(TR^{-1})^T$, which equals the row $\BZ$-span of $TR^{-1}$. We conclude that the choice of the $\CN(T)$ does not depend on the choice of $T$. 
\end{proof}

\section{The P=W phenomenon}
In this section, we will construct for a full rank $m\times n$ ($m\ge n$) integer matrix $M$ a real analytic manifold $\CM$ and endow it with two different structures: an algebraic scheme structure to make it a smooth affine variety $X(M)$ and a Lagrangian fibration structure $h:Y(M)\to \BR^{n+m}$ with general fiber a torus. We show will show that the mix Hodge-theoretic weight filtration on $H^*(X(M),\BQ)$ matches the perverse filtration on $H^*(Y(M),\BQ)$ associated with $h:Y(M)\to \BR^{m+n}$.

\subsection{The space $X(M)$}
Let $M=(a_{ij})_{m\times n}$ be a $m\times n$ matrix with integer coefficients. We define $X(M)$ as a subvariety 
\[
\BC^{2n}\times(\BC^*)^m=\Spec \BC[x_1,\cdots,x_n,y_1,\cdots,y_n,z_1^{\pm1},\cdots,z_m^{\pm1}]
\]
defined by equations
\begin{equation} \label{x1}
x_jy_j=\prod_{i=1}^m z_i^{a_{ij}}+1,~~ 1\le j\le n.
\end{equation}

To simplify the notation, when we represent a point $P\in X(M)$ in terms of the coordinates of $\BC^{2n}\times (\BC^*)^m$, we will simply write $(x_i,y_j,z_k)$ instead of $(x_1,\cdots,x_n,y_1,\cdots,y_n,z_1,\cdots,z_m)$. 

It is straightforward from the definition to see that $X(M)$ is a smooth affine variety of dimension $m+n$. In fact, we have

\begin{prop} \label{3.1}
The affine variety $X(M)$ is isomorphic to the isolated cluster variety $\CA\binom{0_{n\times n}}{M}$ defined in Section 2.1.
\end{prop}

\begin{proof}
Let 
\[
x'_i=y_i\cdot\prod_{i=1}^m z_i^{a_{ij}^-},
\]
where the notation $a^+$ and $a^-$ for $a\in\BR$ are defined in Section 2.1. Then 
\[
x_ix'_i=\prod_{i=1}^m z_i^{a_{ij}+ a_{ij}^-}+\prod_{i=1}^m z_i^{a_{ij}^-}=\prod_{i=1}^m z_i^{a_{ij}^+}+\prod_{i=1}^m z_i^{a_{ij}^-}.
\]
So the invertible morphism $(x_i,y_i,z_k)\mapsto (x_i,x'_i,z_k)$ gives an isomorphism $X(M)\cong \CA\binom{0_{n\times n}}{M}$.
\end{proof}

Since the graph defining any isolated cluster variety has no edges among mutable variables, the top $n\times n$ minor of its extended is zero, and hence is of the form $\binom{0}{M}$. Therefore, by abuse of notations, we also call $X(M)$ the isolated cluster variety associated with integer matrix $M$.

The following results concern maps between $X(M)$ for different $M$'s.    Our statements are slightly stronger than \cite[Proposition 5.8, 5.10]{LS} in the case of isolated cluster varieties. For the convenience of the reader, we give a direct proof.

\begin{prop} \label{cover}
Let $M$ and $M'$ be two $m\times n$ integer matrices such that $M'=TM$ where $T=(t_{ij})$ is an full rank $m\times m$ integer matrix. Then there is a natural finite \'etale morphism between the associated cluster varieties $T_X:X(M')\to X(M)$, defined by 
\begin{equation}\label{x2}
\begin{cases}
x_i=x'_i&1\le i\le n,\\
y_j=y'_j&1\le j\le n,\\
\displaystyle z_k=\prod_{l=1}^m (z'_l)^{t_{lk}}&1\le k\le m.
\end{cases}
\end{equation}
The deck transformation group is the finite abelian group $\CN(T^T)$ (defined in Section 2.3) and in particular, $X(M)=X(M')/\CN(T^T)$. 
\end{prop}

\begin{proof}
Let $M=(a_{ij})$ and $M'=(a'_{ij})$. We first check that (\ref{x2}) gives a morphism $X(M)$. It suffices to check that the defining equations of $X(M')$ is sent to the ones of $X(M)$. In fact,
\[
\begin{split}
\prod_{k=1}^m (z_k)^{a_{kj}}&=\prod_{k=1}^m\left(\prod_{l=1}^m (z'_l)^{t_{lk}}\right)^{a_{kj}}
=\prod_{l=1}^m\prod_{k=1}^m (z'_l)^{t_{lk}a_{kj}}\\
&=\prod_{l=1}^m (z'_l)^{\sum_{k=1}^m t_{lk}a_{kj}}=\prod_{l=1}^m (z'_l)^{a'_{lj}},
\end{split}
\]
as desired.
Let $P=(x_i,y_j,z_k)$ be a point in $X(M)$, and fix one of its preimage $(x_i,y_j,z'_k)\in X(M')$. Then any preimage of $P$ can be written in the form $(x_i,y_j,\lambda_kz'_k)$ since $z'_k\neq0$. Equations (\ref{x2}) imply that the $\lambda_k$'s have to satisfy the equations
\[
\prod_{l=1}^m \lambda_l^{t_{lk}}=1, ~~ 1\le k\le n.
\]
In the notations of Section 2.3, the set of preimages of $P$ is bijective to $\CN(T^T)$.
It is straightforward to see that $X_T$ is an \'etale morphism, as desired.
\end{proof}

\begin{prop} \label{3.3}
Let $M$ be an $m\times n$ integer matrix of full rank, where $m\ge n$. Then there exists an integer $d$ and a full rank $m\times m$ integer matrix $T$ such that $T$ induces a finite \'etale morphism $T_X:X\binom{dI_n}{0}\to X(M)$, and the group of deck transformation of the covering is $\CN(T^T)$ defined in Section 2.3. Concretely, $\CN(T^T)$ is a subgroup of $G\cong(\BZ/d\BZ)^n$, where $G$ consists of automorphisms of $X\binom{dI_n}{0}$ of the type 
\begin{equation}\label{auto}
\begin{cases}
x_i\mapsto x_i & 1\le i\le n\\
y_j\mapsto y_j & 1\le j\le n\\
z_k\mapsto \zeta_k z_k &1\le k\le n\\
\phantom{z_k\mapsto}z_k &n< k\le m 
\end{cases}
\end{equation}
where $\zeta_1,\cdots,\zeta_n$ are $d$-th roots of unity.
\end{prop}

\begin{proof}
By Proposition \ref{lemma}, we obtain full rank integer matrices $\tilde{M}$ and $T$ such that 
\[
\binom{I_n}{0}\xrightarrow{\tilde{M}\cdot} M\xrightarrow{T\cdot} \binom{dI_n}{0}
\]
with composition the multiplication by $\left(\begin{array}{cc}dI_n&0\\0&I_{m-n}\end{array}\right)$. 
By Proposition \ref{cover}, we have \'etale coverings of cluster varieties
\[
X\binom{I_n}{0}\leftarrow X(M)\leftarrow X\binom{dI_n}{0}
\]
The group of deck transformations of $T_X$ is naturally a subgroup of the one of $X\binom{dI_n}{0}\to X\binom{I_n}{0}$. Note that the latter is induced by $\left(\begin{array}{cc}dI_n&0\\0&I_{m-n}\end{array}\right)$, which is exactly of the form (\ref{auto}) by Proposition \ref{cover}.
\end{proof}

\begin{cor} \label{cor}
Let $M$ be a full rank $m\times n$ integer matrix, where $m\ge n$. Let integer $d$ and integer matrix $T$ be obtained in the Proposition \ref{lemma}. Then $X(M)\cong X(d)^{n}/\CN(T^T)\times (\BC^*)^{m-n}$.
\end{cor}

\begin{proof}
By Proposition \ref{3.3}, we have $X(M)=X\binom{dI_n}{0}/\CN(T^T)$. By definition $X\binom{dI_n}{0}=X(d)^n\times (\BC^*)^{m-n}$. By (\ref{auto}), the group $\CN(T^T)\subset (\BZ/d\BZ)^n$ acts on the factor $(\BC^*)^n$ trivially. Therefore, $X(M)\cong X(d)^n/\CN(T^T)\times (\BC^*)^{m-n}$ as desired.
\end{proof}

\begin{rmk}
We remark that the hypothesis $m\ge n$ is crucial in Proposition \ref{3.3} and Corollary \ref{cor}.
In the perspective of the correspondence in Proposition \ref{3.1}, the matrix $B=\binom{0}{M}$ is of full rank if and only if $M$ is of full column rank, which is equivalent to $M$ is of full rank and $m\ge n$. In fact, the geometry of the isolated cluster varieties are quite different when $M$ 
is not of full column rank. For instance, $X(M)$ can not be obtained by a finite group quotient of product of isolated cluster varieties of dimension 1 and 2, and does not satisfy any Lefschetz-type symmetry. 
\end{rmk}

\subsection{The space $Y(M)$}
In this section, we will define a (real) Lagrangian fibration denoted as $Y(M)\to \BR^{n+m}$ with general fiber an $(n+m)$-torus and study its geometric properties.

\begin{defn} \label{dfn}
   Let $M$ be an $m\times n$ integer matrix, where $m\ge n$. We define the fibration $Y(M)\to \BR^{n+m}$ as follows:
   \begin{enumerate}
      \item The total space $Y(M)$ is the underlying real analytic manifold of the complex variety $X(M)$,
      \item The $C^\infty$ surjective map $h:Y(M)\to \BR^{n+m}$ is 
       \begin{equation} \label{eqn}
       h(x_i,y_j,z_k)=(|x_i^2-y_i^2|,\log|z_k|).
       \end{equation}
   \end{enumerate}
\end{defn}

It is obvious from the defining equations (\ref{eqn}) that even when $n+m$ is even, the map $h$ is not a holomorphic if we endow $Y(M)$ the complex structure of $X(M)$ and the natural complex structure on $\BR^{n+m}=\BC^{(n+m)/2}$. We will show that when $M$ is of full rank, the map $f$ is a product of a Lagrangian fibration which admits  K\"ahler structure and a trivial torus fibration. We first recall some facts about local models of elliptic fibrations.

By (local) elliptic fibration we mean a proper holomorphic map $f:Y\to D$ from  complex surface $Y$ to the unit disk $D$ such that the general fibers are elliptic curves. Kodaira's table of singular elliptic fibrations gives a complete classification of local elliptic fibration. For the purpose of this paper, we are mainly interested in the type $I_d$, \emph{i.e.} the central fiber consists of $d$ copies of $\BP^1$ intersecting as a necklace. The following result concerns the relation among elliptic fibrations $I_d$ for different $d$.

\begin{prop} \label{Kahler}
Let $Y_d\to D$ be the elliptic fibration of type $I_d$ over the unit disk $D$. Then there is a properly discontinuous $G=\BZ/d\BZ$-action of isomorphisms of K\"ahler manifolds on $Y$, such that the quotient space $Y/G$ fits into the commutative diagram in the category of K\"ahler manifolds

\[
\begin{tikzcd}
Y_d\arrow[r]\arrow[rd]& Y_d/G=Y_1\arrow[d]\\
 & D & 
\end{tikzcd}
\]
with $Y/G\to D$ the elliptic fibration of type $I_1$. Furthermore, this construction is functorial, \emph{i.e.} when $d$ is a multiple of $d'$, then the quotient map $Y_{d}\to Y_1$ factors through $Y_{d'}\to Y_1$.
\end{prop}

\begin{proof}
This statement essentially follows from \cite[Theorem/Definition 6.1]{DMS}. We first construct the universal cover of the $Y_1$. Start with the trivial projection $D\times \BP^1\to D$, and let $q$ be the coordinate on the base $D$. We iteratively blow-up the two intersections of central fiber $\BP^1\times 0$ and the strict transform of the sections $0\times D$ and $\infty\times D$. Finally remove $0\times D^*$ and $\infty\times D^*$. We denote the fibration we obtained by $q:\mathfrak{D}\to D$. By construction, the central fiber is a chain of infinitely many copies of $\BP^1$ and $q^{-1}(D)\to D$ is a trivial $\BC^*$ bundle. A easy local coordinate calculation indicates that there exists a $\BZ$-action on $\mathfrak{D}$, such that on fiber $\BC^*_t$ over $t\in D$, $n$ acts by multiplication by $t^n$ and on the central fiber $q^{-1}(0)$, $n$ acts by translating the infinite chain by $n$ steps. This action is properly discontinuous, preserves the K\"ahler form, and the quotient $\mathfrak{D}/\BZ\to D$ is exactly the elliptic fibration of type $I_1$. Let $d\BZ\subset \BZ$ be the subgroup, then $\mathfrak{D}/d\BZ\to D$ is the elliptic fibration of type $I_d$. The functoriality follows automatically from the construction.
\end{proof}

\begin{prop}{\cite[Theorem 4.2]{Z}} \label{3.7}
Let $h:Y(d)\to \BR^2$ be the Lagrangian fibration. Then the fibration of real analytic manifolds $h:Y(d)\to \BR^2$ underlies the elliptic fibration of type $I_d$, i.e. there exists K\"ahler structures $(J_d,\omega_d)$ and $(J,\omega)$ on $Y(d)$ and $\BR^2$, respectively such that $h$ is a morphism of K\"ahler manifolds, and the following diagram commutes.
    \[
    \begin{tikzcd}
     (Y(d),J_d,\omega_d)\ar[r,"holom.","\cong"swap]\ar[d]& Y_d\ar[d]\\
     (\BR^2,J,\omega)\ar[r,"holom.","\cong"swap]& D      
    \end{tikzcd}    
    \]
    Furthermore, the K\"ahler structures are compatible with the $\BZ/d\BZ$-action, i.e. the following diagram commutes.
    \[
    \begin{tikzcd}
    (Y(d),J_d,\omega_d)\arrow[r]\arrow[d,"/G"]& Y_d\arrow[d,"/G"]\\
    (Y(1),J_1,\omega_1)\arrow[r]&Y_1,
    \end{tikzcd}
    \]
where $Y(d)\to Y(1)$ is the underlying quotient map by $\BZ/d\BZ$ of real analytic spaces defined by $T=(d)$ in Proposition \ref{cover}, and $Y_d\to Y_1$ is defined in Proposition \ref{Kahler}.
\end{prop}

Now let $M$ be a full rank $m\times n$ integer matrix, where $m\ge n$. Proposition \ref{lemma} produces an integer $d$ and an $m\times m$ full rank integer matrix $T$. Let $Y^n_d\to D^n$ be the $n$-fold self-Cartesian product. Then by Proposition \ref{Kahler}, the total space $Y^n_d$ admits a  $(\BZ/d\BZ)^n$-action of automorphisms of K\"ahler manifolds, which preserves fibers. Note that Proposition \ref{2.7} implies $\CN(T^T)\subset (\BZ/d\BZ)^n$. We now have a proper map between K\"ahler manifolds $Y^n_d/\CN(T^T)\to D^n$.

\begin{prop} \label{3.8}
Let $M$ be a full rank $m\times n$ integer matrix, where $m\ge n$. Let integer $d$ and integer matrix $T$ be obtained in the Proposition \ref{lemma}. Then $h:Y(M)\to \BR^{n+m}$ is the Cartesian product of $Y_d^n/\CN(T^T)\to D^n$ and the trivial torus fibration $(\BC^*)^{m-n}\to \BR^{m-n}$.
\end{prop}

\begin{proof}
By Definition \ref{dfn}, the fibration $Y\binom{dI_n}{0}\to \BR^{n+m}$ is the product of $Y(d)^n\to \BR^{2n}$ and $(\BC^*)^{m-n}\to \BR^{m-n}$. By the concrete description (\ref{auto}), the natural $(\BZ/n\BZ)^n$-action on $Y\binom{dI_n}{0}$ preserves the fibers of $Y(d)^n\to \BR^{2n}$ and is trivial on the factor $(\BC^*)^{m-n}\to \BR^{m-n}$. Now forgetting the complex structure of Corollary \ref{cor}, we have $Y(M)=Y(d)^n/\CN(T^T)\times (\BC^*)^{m-n}\to \BR^{n+m}$ is a Cartesian product of $Y(d)^n/\CN(T^T)\to\BR^{2n}$ and $(\BC^*)^{m-n}\to \BR^{m-n}$. Finally by Proposition \ref{3.7},  $Y(d)^n/\CN(T^T)\to \BR^{2n}$ underlies the holomorphic Lagrangian fibration $Y_d^n/\CN(T^T)\to D^n$.
\end{proof}

The smooth structure of general fibers of $h:Y(M)\to \BR^{n+m}$ easily follows from Proposition \ref{3.8}.

\begin{cor}
   Let $M$ be a full rank $n\times m$ integer matrix with $m\ge n$. Then the general fibers of $h:Y(M)\to \BR^{n+m}$ are torus $T^{n+m}$.
\end{cor}

\subsection{P=W}
Let $M$ be a full rank $m\times n$ integer matrix with $m\ge n$. Let $X(M)$ be the associated isolated cluster variety and $Y(M)\to \BR^{n+m}$ be the Lagrangian fibration. In this section, we will show that the mixed Hodge-theoretic weight filtration on the cohomology groups of $X(M)$ equals the perverse filtration associated with the fibration $h:Y(M)\to \BR^{n+m}$, \emph{i.e.} the $P=W$ identity holds.

We start with a trivial case: $X=\BC^*$ and $Y=\BC^*\cong S^1\times \BR\to\BR$, which corresponds to the degenerate matrix $m=1$, $n=0$.

\begin{lem} \label{trivial}
   Let $h:\BC^*\cong S^1\times \BR\to \BR$ be the natural projection. Then $h$ admits a perverse decomposition in the sense of equation (\ref{0101}) and the $P=W$ identity holds for the quasi-projective variety $\BC^*$ and fibration $h$.
\end{lem} 

\begin{proof}
   The decomposition of the projection $h$ is clear   
   \begin{equation}\label{0202}
   Rh_*\BQ_{\BC^*}=\BQ_{\BR^1}\oplus \BQ_{\BR^1}[-1].
   \end{equation}
   With respect to the choice of the perversity function, $\BQ_{\BR^1}$ is a perverse sheaf. Therefore, equation $(\ref{0202})$ is a perverse decomposition of the fibration $h$. The perverse filtration is exactly the cohomological filtration, \emph{i.e.} 
   \[
   P_kH^*(\BC^*,\BQ)=H^{\le k}(\BC^*,\BQ),~~~~ k=0,1.
   \]
   On the other hand, the weight filtration satisfies $W_{2k}H^*(\BC^*,\BQ)=H^{\le k}(\BC^*,\BQ)$. We conclude that the $P=W$ identity holds.
\end{proof}

The $m=n=1$ case is proved as one of the main theorems of \cite{Z}.

\begin{prop} {\cite[Theorem 5.5]{Z}} \label{base}
    Let $d$ be a non-zero integer. Then the $P=W$ identity holds for the isolated cluster variety $X(d)$ and the Lagrangian fibration $Y(d)\to \BR^{n+m}$.
\end{prop}

\begin{thm} \label{thm}
   Let $M$ be a full rank $m\times n$ integer matrix, where $m\ge n$. Then the $P=W$ identity holds for the isolated cluster variety $X(M)$ and the Lagrangian fibration $Y(M)\to \BR^{n+m}$, i.e.
    \[
    P_kH^*(Y(M),\BQ)=W_{2k} H^*(X(M),\BQ)=W_{2k+1}H^*(X(M),\BQ),~~ k\ge 0. 
    \]
\end{thm}

\begin{proof}
   Given a full rank $m\times n$ integer matrix $M$, Proposition \ref{lemma} produces an integer $d$ and a full rank $m\times m$ integer matrix $T$ which are used in Corollary \ref{cor} and Proposition \ref{3.8}. 
   To simplify notations, we write $\Gamma=\CN(T^T)$. Then by Corollary \ref{cor}, $X(M)=X(d)^{\times n}/\Gamma\times (\BC^*)^{m-n}$. Then the cohomology group decomposes as
   \[
        H^*(X(M),\BQ)\cong \left(H^*(X(d),\BQ)^{\otimes n}\right)^{\Gamma}\otimes H^*(\BC^*,\BQ)^{\otimes m-n}
   \]
   Since $\Gamma$ acts by automorphisms of algebraic varieties, it preserves the weight filtration, and hence 
   \[
   W_k(H^*(X(d),\BQ)^{\otimes n})^\Gamma=(W_kH^*(X(d),\BQ)^{\otimes n})^\Gamma.
   \]
   Since the K\"unneth decomposition respects mixed Hodge structures (see for example \cite[Example 3.2, Theorem 5.44]{PS}), we have 
   \[
   \begin{split}
   W_kH^*(X(M),\BQ)\cong &\sum_{t+t'=k}(W_t(H^*(X(d),\BQ)^{\otimes n}))^\Gamma\otimes W_{t'}(H^*(\BC^*,\BQ)^{\otimes m-n}) \\
   \cong&\sum_{t_1+\cdots+t_{n+m}=k}\left(\bigotimes_{i=1}^n W_{t_i}H^*(X(d),\BQ)\right)^\Gamma\bigotimes \bigotimes_{j=n+1}^{m} W_{t_j}H^*(\BC^*,\BQ). 
   \end{split}
   \]
   
   On the other hand, by Proposition \ref{3.8}, the proper fibration $h:Y(M)\to \BR^{n+m}$ is the Cartesian product of $Y_d^n/\Gamma\to D^n$ and the trivial torus fibration $(\BC^*)^{m-n}\to \BR^{m-n}$. By Proposition \ref{Kahler}, the fibration $Y_d^n/\Gamma\to D^n$ is a proper morphism of K\"ahler varieties, hence admits a perverse decomposition in the sense of equation (\ref{0101}) \cite[Theorem 0.6]{S}. Together with Lemma \ref{trivial}, we conclude that the K\"unneth decomposition respects the perverse filtration associated with $h$:  
   \[
   \begin{split}
   P_kH^*(Y(M),\BQ)\cong &\sum_{t+t'=k}\left(P_t(H^*(Y(d),\BQ)^{\otimes n})\right)^\Gamma\otimes P_{t'}(H^*(\BC^*,\BQ)^{\otimes m-n}) \\
    \textup{(Proposition \ref{quotient})}\cong&\sum_{t+t'=k}P_t\left((H^*(Y(d),\BQ)^{\otimes n})^\Gamma\right)\otimes P_{t'}(H^*(\BC^*,\BQ)^{\otimes m-n})\\ 
   \cong&\sum_{t_1+\cdots+t_{n+m}=k}\left(\bigotimes_{i=1}^n P_{t_i}H^*(Y(d),\BQ)\right)^\Gamma\bigotimes \bigotimes_{j=n+1}^{m} P_{t_j}H^*(\BC^*,\BQ). 
   \end{split}
   \]
   By Lemma \ref{trivial} and Proposition \ref{base}, we conclude that the $P=W$ identity holds for the isolated cluster variety $X(M)$ and the Lagrangian fibration $h:Y(M)\to \BR^{n+m}$.

\end{proof}

\end{document}